\documentclass[11pt]{article}%
\usepackage{a4}%
\usepackage[centertags]{amsmath}%
\usepackage{eucal,dsfont,accents,bbm,amsfonts,amssymb,amsthm,amsopn}%
\usepackage[usenames]{color}
\definecolor{citegreen}{rgb}{0,0.6,0}
\definecolor{refred}{rgb}{0.8,0,0}
\usepackage[colorlinks, citecolor=citegreen, linkcolor=refred]{hyperref}
\title{Dynamical stability and instability of Ricci-flat metrics}
\author{Robert Haslhofer and Reto M\"{u}ller}
\date{}
\providecommand{\abs}[1]{\lvert #1\rvert}%
\providecommand{\norm}[1]{\lVert #1\rVert}%
\DeclareMathOperator{\Hess}{Hess}%
\DeclareMathOperator{\divop}{div}%
\DeclareMathOperator{\id}{id}%
\DeclareMathOperator{\Rm}{Rm}%
\DeclareMathOperator{\Rc}{Rc}%

\newcommand{\RR}{\mathbb{R}}%
\newcommand{\eps}{\varepsilon}%
\newcommand{\Lap}{\triangle}%
\newcommand{\D}{\nabla}%
\newcommand{\sS}{\mathcal{S}}%
\newcommand{\U}{\mathcal{U}}%
\newcommand{\V}{\mathcal{V}}%
\newcommand{\B}{\mathcal{B}}%
\newcommand{\dt}{\partial_t}%
\newtheoremstyle{break}%
  {12pt}%
  {16pt}%
  {\itshape}%
  {}%
  {\bfseries}%
  {}%
  {\newline}%
  {\thmname{#1}\thmnumber{ #2}\thmnote{ \normalfont{(#3)}}}%

\theoremstyle{definition}%
\theoremstyle{remark}%
\newtheorem*{rem}{Remark}%
\theoremstyle{break}%
\newtheorem{lemma}{Lemma}[section]%
\newtheorem{thm}[lemma]{Theorem}%
\numberwithin{equation}{section}%
\begin{document}%
\maketitle%
\pagenumbering{arabic}%
\begin{abstract}
In this short article, we improve the dynamical stability and instability results for Ricci-flat metrics under Ricci flow proved by Sesum \cite{Ses06} and Haslhofer \cite{Has1}, getting rid of the integrability assumption.
\end{abstract}

\section{Introduction}

Let $M$ be a compact manifold. A Ricci-flat metric on $M$ is a Riemannian metric with vanishing Ricci curvature.
Ricci-flat metrics are fairly hard to construct, and their properties are of great interest, see \cite{Bes,Joy,LB} for extensive information.
They are the critical points of the Einstein-Hilbert functional, $\mathcal{E}(g)=\int_{M} R_g dV_g$, the fixed points of Hamilton's Ricci flow \cite{Ham_surv}, 
\begin{equation}\label{eq.Ricciflow}
\dt g(t)=-2\Rc_{g(t)},
\end{equation}
and the critical points of Perelman's $\lambda$-functional \cite{Per1},
\begin{equation}\label{deflambda}
\lambda(g)=\inf_{\substack{f\in C^\infty(M)\\ \int_M e^{-f}dV_g=1}}\int_M \big(R_g+\abs{\D f}_g^2\big)e^{-f}dV_g.
\end{equation}
In this article, we are concerned with the \emph{stability properties} of Ricci-flat metrics under Ricci flow. This stability problem has been studied previously by Sesum \cite{Ses06} and Haslhofer \cite{Has1}, generalizing in turn previous work by Guenther-Isenberg-Knopf \cite{GIK}. The main theorems established there are the \emph{dynamical stability theorem} \cite[Thm. 3]{Ses06}, \cite[Thm. E]{Has1}, and the \emph{dynamical instability theorem} \cite[Thm. F]{Has1}. The dynamical stability theorem says that if a Ricci-flat metric is a local maximizer of $\lambda$ and if all its infinitesimal Ricci-flat deformations are integrable, then every Ricci flow starting close to it exists for all times and converges (modulo diffeomorphisms) to a nearby Ricci-flat metric.
The dynamical instability theorem says that if a Ricci-flat metric is not a local maximizer of $\lambda$ and if all its infinitesimal Ricci-flat deformations are integrable, then there exists a nontrivial ancient Ricci flow emerging from it.
However, the integrability assumption is rather strong and it is natural to ask whether or not this assumption can be weakened or even removed. In the present article, we prove that it is indeed possible to completely remove this integrability assumption imposed by Sesum and Haslhofer, i.e.~we prove the following.

\begin{thm}[Dynamical stability]\label{thm.stability}
Let $(M,\hat{g})$ be a compact Ricci-flat manifold. If $\hat{g}$ is a local maximizer of $\lambda$, then for every $C^{k,\alpha}$-neighborhood $\U$ of $\hat{g}$ $(k\geq 2)$, there exists a $C^{k,\alpha}$-neighborhood $\V\subset\U$ such that the Ricci flow starting at any metric in $\V$ exists for all times and converges (modulo diffeomorphisms) to a Ricci-flat metric in $\U$.
\end{thm}

\begin{thm}[Dynamical instability]\label{thm.instability}
Let $(M,\hat{g})$ be a compact Ricci-flat manifold. If $\hat{g}$ is not a local maximizer of $\lambda$, then there exists a nontrivial ancient Ricci flow $\{g(t)\}_{t\in(-\infty,0]}$ that converges (modulo diffeomorphisms) to $\hat{g}$ for $t\to -\infty$.
\end{thm}

Theorems \ref{thm.stability} and \ref{thm.instability} describe the dynamical behavior of the Ricci flow near a given Ricci-flat metric. In fact, they show that dynamical stability and instability are characterized exactly by the local maximizing property of $\lambda$, observing whether or not $\lambda\leq 0$ in some $C^{k,\alpha}$-neighborhood of $\hat{g}$ ($k\geq 2$). Indeed, the converse implications follow immediately from Perelman's monotonicity formula \cite{Per1}, i.e. if the conclusion of Theorem \ref{thm.stability} holds, then $\hat{g}$ it is a local maximizer of $\lambda$; if the conclusion of Theorem \ref{thm.instability} holds, then $\hat{g}$ is not a local maximizer of $\lambda$.

\begin{rem}
Another related notion is \emph{linear stability}, meaning that all eigenvalues of the Lichnerowicz Laplacian $L_{\hat{g}}=\Lap_{\hat{g}}+2\Rm_{\hat{g}}$ are nonpositive. If $\hat{g}$ is a local maximizer of $\lambda$, then it is linearly stable \cite[Thm. 1.1]{CHI04}. If $\hat{g}$ is linearly stable and integrable, then it is a local maximizer of $\lambda$, c.f. \cite[Thm. A]{Has1}.
\end{rem}

In addition to applying to the more general nonintegrable case, the proofs that we give here are substantially shorter than the previous arguments from \cite{Ses06,Has1}. In outline, we start by proving the following {\L}ojasiewicz-Simon inequality for Perelman's $\lambda$-functional, which generalizes \cite[Thm. B]{Has1} to the nonintegrable case:

\begin{thm}[{\L}ojasiewicz-Simon inequality for $\lambda$]\label{thm.LS}
Let $(M,\hat{g})$ be a closed Ricci-flat manifold. Then there exists a $C^{2,\alpha}$-neighborhood $\U$ of $\hat{g}$ in the space of metrics on $M$ and a $\theta\in (0,\frac{1}{2}]$, such that
\begin{equation}\label{eq.LS}
\norm{\Rc_g+\Hess_g f_g}_{L^2(M,e^{-f_g} dV_g)}\geq \abs{\lambda(g)}^{1-\theta}\, ,
\end{equation}
for all $g\in\U$, where $f_g$ is the minimizer in \eqref{deflambda} realizing $\lambda(g)$.
\end{thm}

Theorem \ref{thm.LS} can be used as a general tool to study stability and convergence questions for the Ricci flow, and might thus be of independent interest. A key step in our proof of Theorem \ref{thm.stability} is then to modify the Ricci flow by an appropriate family of diffeomorphisms so that we can on the one hand exploit the geometric inequality (\ref{eq.LS}) and on the other hand retain the needed analytic estimates.
This is quite related to the proof by Sun-Wang of the stability of positive K\"{a}hler-Einstein metrics under the normalized K\"{a}hler-Ricci flow \cite{SW10}; the details about handling the diffeomorphism group are somewhat different, however.
The proof of Theorem \ref{thm.instability} is related to the proof of \cite[Thm. F]{Has1}, again with some modifications.

\paragraph{Acknowledgements.} We thank Felix Schulze and Leon Simon for discussions on related issues. RM was financially supported by an Imperial College Junior Research Fellowship.

\section{The {\L}ojasiewicz-Simon inequality for $\lambda$}

\begin{proof}[Proof of Theorem \ref{thm.LS}]
By the Ebin-Palais slice theorem \cite{Ebi70}, there exists a $C^{2,\alpha}$-neighborhood $\U$ of $\hat{g}$ in the space of metrics on $M$ and a $\sigma>0$ such that every metric in $\U$ can be written as the pullback of a metric in the slice
\begin{equation}
\sS_{\hat{g}}:=\{{\hat{g}}+h | h\in\ker\divop_{\hat{g}}, \norm{h}_{C^{2,\alpha}} < \sigma \} .
\end{equation}
Since both sides of (\ref{eq.LS}) are diffeomorphism invariant, it thus sufficies to prove the inequality (\ref{eq.LS}) for the metrics $g$ in the slice $\sS_{\hat{g}}$.\\

The proof is now, with a couple of little tweaks, along the lines of the classical proof due to Leon Simon \cite{Sim83}. For our purpose it is most convenient to apply the variant of Simon's theorem that can be found in Colding-Minicozzi \cite[Thm. 6.3]{CM12}.
To apply this theorem, we have to observe that the restricted functional $\lambda_R:\sS_{\hat{g}}\to \RR$ is analytic and that its gradient and its Hessian satisfy certain properties:\\

(1) As pointed out by Perelman \cite{Per1}, by substituting $w=e^{-f/2}$ in (\ref{deflambda}) one sees that $\lambda(g)$ is the smallest eigenvalue of the Schr{\"o}dinger operator $-4\Lap_g+R_g$. Since the smallest eigenvalue is simple, $\lambda$ depends analytically on $g$, c.f. \cite{RS78}. In particular, the restriced functional $\lambda_R$ is analytic.\\

(2) The $L^2(M,e^{-f_g}dV_g)$-gradient of $\lambda$ is $\nabla\lambda(g)=-(\Rc_g+\Hess_g f_g$) by Perelman's first variation formula \cite{Per1}. Recall that $\nabla\lambda(\hat{g})=0$. Furthermore, note that the gradient satisfies the estimates
\begin{align}
 \norm{\nabla\lambda(g_1)-\nabla\lambda(g_2)}_{C^{0,\alpha}}&\leq C\norm{g_1-g_2}_{C^{2,\alpha}}\label{gradest1}\,\\
\norm{\nabla\lambda(g_1)-\nabla\lambda(g_2)}_{L^2}&\leq C\norm{g_1-g_2}_{W^{2,2}}\label{gradest2}
\end{align}
for all $g_1,g_2\in \sS_{\hat{g}}$. Indeed, these estimates follow from basic elliptic theory using the eigenvalue equation $(-4\Lap_g+R_g)e^{-f_g/2}=\lambda(g)e^{-f_g/2}$.
Since the gradient of the restriced functional can be obtained from the gradient of the unrestricted functional by projecting to the slice, the estimates (\ref{gradest1}) and (\ref{gradest2}) also hold for $\nabla\lambda_R$.\\

(3) The linearization of $\nabla\lambda_R$ at $\hat{g}$ is given by $L_{\hat{g}}=\frac{1}{2}\Lap_{\hat{g}}+\Rm_{\hat{g}}$, see e.g. \cite[Lem. 4.2]{Has1}.
We write $(\ker\divop_{\hat{g}})_{C^{k,\alpha}}$ for the space of $C^{k,\alpha}$ symmetric two tensors with vanishing divergence.
By ellipticity, the operator
\begin{equation}
L_{\hat{g}}:(\ker\divop_{\hat{g}})_{C^{2,\alpha}}\to (\ker\divop_{\hat{g}})_{C^{0,\alpha}}
\end{equation}
is Fredholm. Clearly, it also satisfies the estimate $\norm{L_{\hat{g}}h}_{L^2}\leq C \norm{h}_{W^{2,2}}$.\\

We can now apply \cite[Thm. 6.3]{CM12} and obtain that $\lambda_R$ satisfies the inequality $\norm{\nabla\lambda_R}\geq \abs{\lambda_R}^{1-\theta}$ for some $\theta\in(0,\frac{1}{2}]$. Together with $\norm{\nabla\lambda_R}\leq\norm{\nabla\lambda}$ and the reduction due to diffeomorphism invariance from the first paragraph this proves the theorem.
\end{proof}

\section{Dynamical stability and instability}

\begin{proof}[Proof of Theorem \ref{thm.stability}]
We write $\B_r$ for the $C^{k,\alpha}$-ball of radius $r$ around $\hat{g}$. Let $\eps>0$ such that $\B_\eps\subset \U$, $\lambda\leq 0$ in $\B_\eps$, and (\ref{eq.LS}) holds in $\B_\eps$.
We will choose $\V=\B_\delta$, where $\delta\ll\eps$ is small enough such that everything in the following works.\\

Given any $g_0\in \B_\delta$, let $T\in(0,\infty]$ be the maximal time such that the solution $g(t)$ of the Ricci flow \eqref{eq.Ricciflow} starting at $g_0$ exists for $t\in[0,T)$ and there exists a family of diffeomorphisms $\psi_t$ such that $\psi_t^\ast g(t)\in \B_\eps$ for all $t\in[0,T)$.\\

Choosing $\delta$ small enough we can assume that the Ricci-DeTurck flow (see e.g. \cite{Ham_surv}) stays in $\B_{\eps/4}$ up to time one;
in particular $T\geq 1$ and there exists a diffeomorphisms $\psi_1$ such that $\psi_1^\ast g(1)\in\B_{\eps/4}$.\\

By the definition of $T$ we have uniform curvature bounds
\begin{equation}\label{anest1}
\sup_M\abs{\Rm_{g(t)}}_{g(t)}\leq C, \qquad \forall t\in[0,T).
\end{equation} 
By standard derivative estimates (see e.g. \cite{Ham_surv}) this implies
\begin{equation}\label{anest2}
\sup_M\abs{\nabla^\ell\Rm_{g(t)}}_{g(t)}\leq C_\ell, \qquad \forall t\in[1,T).
\end{equation} 
Since $f_{g}$ solves the elliptic equation $(-4\Lap_{g}+R_g-\lambda(g))e^{-f_g/2}=0$ we also get
\begin{equation}\label{anest3}
\sup_M\abs{\nabla^\ell f_{g(t)}}_{g(t)}\leq \tilde{C}_\ell, \qquad \forall t\in[1,T).
\end{equation}
Note that the estimates (\ref{anest1}), (\ref{anest2}) and (\ref{anest3}) are diffeomorphism invariant.\\

Define $\{\varphi_t:M\to M\}_{t\in[1,T)}$ to be the family of diffeomorphisms generated by $X(t)=-\D f_{\psi_1^\ast g(t)}$ with $\varphi_1=\id_M$ and let $\{\psi_t=\psi_1\circ\varphi_t\}_{t\in[1,T)}$ with $\psi_1$ from above. Then the pulled-back metrics $\{\tilde{g}(t):=\psi_t^\ast g(t)\}_{t\in[1,T)}$ satisfy
\begin{equation}\label{eq.pullbackflow}
\dt \tilde{g}(t)=-2\big(\Rc_{\tilde{g}(t)}+\Hess_{\tilde{g}(t)}f_{\tilde{g}(t)}\big),
\end{equation}
with $\tilde{g}(1)\in\B_{\eps/4}$. Let $T'\in[1,T]$ be the maximal time such that $\tilde{g}(t)\in\B_\eps$ for all $t\in[1,T')$.
By interpolation, using the bounds (\ref{anest2}) and (\ref{anest3}), we obtained
\begin{equation}\label{interpol}
 \norm{\partial_t\tilde{g}}_{C^{k,\alpha}}\leq C \norm{\partial_t\tilde{g}}_{L^2}^{1-\eta}
\end{equation}
for $\eta>0$ as small as we want; in particular, we can assume $\sigma:=\theta-\eta+\theta\eta>0$. Using Perelman's monotonicity formula in the form
\begin{equation}
\frac{d}{dt} \lambda(\tilde{g}(t))=\norm{\Rc_{\tilde{g}(t)}+\Hess_{\tilde{g}(t)}f_{\tilde{g}(t)}}_{L^2}^{1+\eta}\,\norm{\dt \tilde{g}(t)}_{L^2}^{1-\eta},
\end{equation}
as well as the assumption $\lambda(\tilde{g}(t))\leq 0$, we obtain from Theorem~\ref{thm.LS} and \eqref{interpol}
\begin{equation}
\begin{split}
-\frac{d}{dt} \abs{\lambda(\tilde{g}(t))}^{\sigma}&=\sigma \abs{\lambda(\tilde{g}(t))}^{\sigma-1}\cdot\frac{d}{dt}\lambda(\tilde{g}(t))\\
&=\sigma \abs{\lambda(\tilde{g}(t))}^{(\theta-1)(1+\eta)}\cdot\norm{\Rc_{\tilde{g}(t)}+\Hess_{\tilde{g}(t)}f_{\tilde{g}(t)}}_{L^2}^{1+\eta}\cdot\norm{\dt \tilde{g}(t)}_{L^2}^{1-\eta}\\
&\geq \frac{\sigma}{C} \norm{\dt \tilde{g}(t)}_{C^{k,\alpha}}.
\end{split}
\end{equation}
Hence, by integration,
\begin{equation}\label{eq.disttrav}
\int_1^{T'}\norm{\dt \tilde{g}(t)}_{C^{k,\alpha}}dt \leq \tfrac{C}{\sigma}\abs{\lambda(\tilde{g}(1))}^{\sigma} \leq \tfrac{C}{\sigma}\abs{\lambda(g_0)}^{\sigma}\leq \frac{\eps}{4},
\end{equation}
provided we choose $\delta$ small enough. Thus, $T'=T=\infty$ and $\tilde{g}(t)$ converges in $C^{k,\alpha}$ to a limit $g_{\infty}$ for $t\to\infty$. Since $\nabla\lambda(g_\infty)=0$, the limit $g_\infty$ is Ricci-flat. By construction it is a limit modulo diffeomorphism of the Ricci flow.
\end{proof}

\begin{proof}[Proof of Theorem \ref{thm.instability}]
Pick a sequence of metrics $g_i\to \hat{g}$ in $C^{k,\alpha}$ with $\lambda(g_i)>0$.
Let $g_i(t)$ be the Ricci flow starting at $g_i$ and let $\tilde{g}_i(t)=\psi_t^\ast g_i(t)$ with $\{\psi_t\}$ as in the proof of Theorem \ref{thm.stability}, i.e. for $t\in[0,1]$ the pulled-back metrics solve the Ricci-DeTurck flow, while for $t\geq1 $ they solve the modified flow (\ref{eq.pullbackflow}).
Then we still have $\tilde{g}_i:=\tilde{g}_i(1)\to \hat{g}$ in $C^{k,\alpha}$ and $\lambda(\tilde{g}_i)>0$, but also uniform estimates for all higher derivatives of $\tilde{g}_i$ (with respect to the background metric $\hat{g}$). Thus, after passing to a subsequence, we get $\tilde{g}_i\to \hat{g}$ in $C^{\ell,\alpha}$ for $\ell\gg k$ as large as we want.
Moreover, for $t\geq 1$ we have uniform estimates as in (\ref{anest2}) and (\ref{anest3}).\\

Let $\eps$ be small enough such that (\ref{eq.LS}) holds in the $C^{k,\alpha}$-ball $\B_{2\eps}$ around $\hat{g}$.
On the one hand, $\tilde{g}_i(t)$ stays close to $\hat{g}$ for longer and longer times, but on the other hand, since $\lambda(g_i)>0$ the Ricci flow becomes singular eventually by Perelman's evolution inequality $\tfrac{d\lambda}{dt}\geq\tfrac{2}{n}\lambda^2$ \cite{Per1}.
Let $t_i$ be the first time when $\norm{\tilde{g}_i(t_i)-\hat{g}}_{C^{\ell,\alpha}}=\eps$. Then $t_i\to\infty$ and, always assuming $i$ is large enough,
\begin{equation}
\norm{\tilde{g}_i(t_i)-\hat{g}}_{C^{\ell,\alpha}}
\leq C\lambda(\tilde{g}_i(t_i))^{\sigma},
\end{equation}
for some $\sigma>0$, which is obtained by the same reasoning as in the proof of Theorem \ref{thm.stability}, using in particular the {\L}ojasiewicz-Simon inequality (\ref{eq.LS}).\\

Shifting time, we obtain a family of flows $\{\tilde{g}^s_i(t):=\tilde{g}_i(t+t_i)\}_{t\in[T_i,0]}$, with $T_i=1-t_i\to-\infty$, that solve the equation (\ref{eq.pullbackflow}) and satisfy
\begin{align}
&\norm{\tilde{g}^{s}_i(t)-\hat{g}}_{C^{\ell,\alpha}}\leq \eps \quad\forall t\in[T_i,0],\label{cinftybounds}\\
&\lambda(\tilde{g}^s_i(0))\geq c>0,\\
&\tilde{g}^s_i(T_i)\to \hat{g} \quad\text{in}\; C^{\ell,\alpha}.
\end{align}
After passing to a subsequence, $\{\tilde{g}^s_i(t)\}$ converges in $C^{k,\alpha}_{\mathrm{loc}}(M\times (-\infty,0])$ to an ancient solution $\{\tilde{g}(t)\}_{t\in(-\infty,0]}$ of (\ref{eq.pullbackflow}).
Define $\{\varphi_t:M\to M\}_{t\in(-\infty,0]}$ to be the family of diffeomorphisms generated by $X(t)=\D f_{\tilde{g}(t)}$ with $\varphi_0=\id_M$. Then $\{g(t):=\varphi_t^\ast \tilde{g}(t)\}$ is an ancient Ricci flow.
Since $\lambda(g(0))\geq c$, this Ricci flow is nontrivial, i.e. not a stationary solution.\\

Finally, we have to prove that for $t\to -\infty$ the Ricci flow $\{g(t)\}$ converges modulo diffeomorphism to $\hat{g}$. To this end, for $T_i\leq t$ using again in particular (\ref{eq.LS}) we estimate
\begin{align}
\norm{\hat{g}-\tilde{g}(t)}_{C^{k,\alpha}}&\leq \norm{\hat{g}-\tilde{g}^s_i(T_i)}_{C^{k,\alpha}}+\norm{\tilde{g}^s_i(T_i)-\tilde{g}^s_i(t)}_{C^{k,\alpha}}+\norm{\tilde{g}^s_i(t)-\tilde{g}(t)}_{C^{k,\alpha}}\nonumber\\
&\leq \norm{\hat{g}-\tilde{g}^s_i(T_i)}_{C^{k,\alpha}}+C\lambda({g_i(t+t_i)})^\sigma+\norm{\tilde{g}^s_i(t)-\tilde{g}(t)}_{C^{k,\alpha}}.
\end{align}
Since $\lambda(g_i(t+t_i))$ is bounded up to $t=0$ and $\tfrac{d\lambda}{dt}\geq\tfrac{2}{n}\lambda^2$, we see that $\lambda(g_i(t+t_i))$ is very small for very negative $t$. Thus, $\tilde{g}(t)=(\varphi_t^{-1})^\ast g(t) \to \hat{g}$ in $C^{k,\alpha}$ as $t\to-\infty$ and this finishes the proof of the theorem.
\end{proof}

%------------------------------------------------
\makeatletter
\def\@listi{%
  \itemsep=0pt
  \parsep=1pt
  \topsep=1pt}
\makeatother
{\fontsize{9}{10}\selectfont
}
\vspace{3mm}

Robert Haslhofer\\
{\sc Courant Institute of Mathematical Sciences, New York University, 251 Mercer Street, New York, NY 10012, USA}\\
\vspace{-1mm}

Reto M\"uller\\
{\sc Department of Mathematics, Imperial College London, 180 Queen's Gate, London SW7 2AZ, United Kingdom}\\

\end{document}